\documentclass[11pt]{article}
\usepackage[utf8]{inputenc}
\usepackage[T1]{fontenc}
\usepackage{comment}
\usepackage{authblk}
\usepackage{relsize}
\usepackage{changepage}
\usepackage{mathtools}
\usepackage[english]{babel}
\usepackage{amsmath}
\usepackage{amsthm}
\usepackage{amsfonts}
\usepackage{amssymb}
\usepackage{graphicx}
\usepackage[usenames,dvipsnames]{color}
\theoremstyle{plain}
\newtheorem{theorem}{Theorem}[section]
\newtheorem{corollary}[theorem]{Corollary}
\newtheorem{proposition}[theorem]{Proposition}
\newtheorem{lemma}[theorem]{Lemma}

\makeatletter
\newcommand{\vast}{\bBigg@{4}}
\newcommand{\Vast}{\bBigg@{5}}
\makeatother
\definecolor{bulgarianrose}{rgb}{0.28, 0.02, 0.03}
\definecolor{gray}{rgb}{0.5, 0.5, 0.5}

\theoremstyle{definition}

\theoremstyle{remark}

\usepackage[left=2cm,right=2cm,top=2cm,bottom=2cm]{geometry}
\usepackage{amssymb}
\makeatletter
\def\namedlabel#1#2{\begingroup
    #2%
    \def\@currentlabel{#2}%
    \phantomsection\label{#1}\endgroup
}
\usepackage[normalem]{ulem}
\usepackage{dsfont}
\usepackage{accents}
\usepackage{verbatim}
\usepackage{ulem}
\usepackage{pgf,tikz,pgfplots}
\pgfplotsset{compat = 1.16}%take out before submission
\usepackage[all]{xy} 

\usepackage{stackengine}
\stackMath
\newcommand\tsup[2][2]{%
 \def\useanchorwidth{T}%
  \ifnum#1>1%
    \stackon[-.5pt]{\tsup[\numexpr#1-1\relax]{#2}}{\scriptscriptstyle\sim}%
  \else%
    \stackon[.5pt]{#2}{\scriptscriptstyle\sim}%
  \fi%
}

\usepackage{enumerate}

\usepackage{amsmath}
\usepackage{amsfonts}
\usepackage{amsthm}
\usepackage{amssymb}
\usepackage{array}
\usepackage{booktabs}
\usepackage{color}
\usepackage{changepage}
\usepackage[english]{babel}
\usepackage{enumerate}
\usepackage{epsfig}
\usepackage{float}
\usepackage{framed}
\usepackage{gensymb}
\usepackage{graphicx}
\usepackage{hyperref}
\usepackage{indentfirst}
\usepackage{longtable}
\usepackage{mathtools}
\usepackage{multirow}
\usepackage[normalem]{ulem}
\usepackage{setspace}
\usepackage{subcaption}
\usepackage{textcomp}
\usepackage[utf8]{inputenc}
\usepackage{verbatim} 
\usepackage{enumerate}
\usepackage{xlop}
\usepackage{listings}
\def\and{% 
  \end{tabular}%
  \hskip 1em \@plus.17fil%
  \begin{tabular}[t]{c}}% 
  
\title{\scshape
  Exponentially sized pointsets with angles less than 61 degrees}

\author{Miroslav Marinov}
\affil{{\tt m.marinov1617@gmail.com}}

\date{}
\begin{document}

\maketitle

%%%%%%%%%%%%%%%%%%%%%%%%%%%%%%%%%%%%%%%%%%%%%%%%%%%%%%%%%%%%%%%%%%%%%%%%%%%%%%%%
\begin{abstract}
    We prove that any maximal set of points in $\mathbb{R}^d$, any three of which form an angle less than $\frac{\pi}{3} + c$, has size $(1+\Theta(c))^d$ for sufficiently small $c>0$. The proof is based on a refinement of an approach by Erd\H{o}s and F\"{u}redi. The lower bound is relying on a problem about large uniform hypergraphs with small edge intersections, while the upper bound is tightly connected to the problem of packing disjoint caps on a sphere.
\end{abstract}

\hspace{1em}Keywords: exponential size, set of points, acute angles

\hspace{1em}2020 Mathematics Subject Classification: 51M04, 52C10

\section{Introduction}
Around $1950$ Erd\H{o}s conjectured that $2^d$ is the greatest possible number of points in $\mathbb{R}^d$ such that no three of them determine an obtuse angle. This was proven in $1962$ by Danzer and Gr\"{u}nbaum (\cite{DG}, \cite{AZ}). The only examples attaining the bound turn out to be affine images of the hypercube. As it contains a lot of angles equal to $\pi/2$, the following question becomes interesting: what is the maximum number of points in $\mathbb{R}^d$ with all angles strictly less than $\pi/2$? Very recently, in 2017, Gerensc\'{e}r and Harangi \cite{GH} made a breakthrough by providing a surprisingly simple construction of $2^{d-1} + 1$ points. Regarding an upper bound, currently it is not known even whether sets of $2^{d} - 1$ points are possible.

The above problem can be naturally generalized by considering, for a positive integer $d \geq 2$ and a real number $\alpha \in (0,\pi)$, the maximal cardinality $g_{\alpha}(d)$  of a set of points in $\mathbb{R}^d$ such that any angle formed by three of the points is \textit{less than or equal to} $\alpha$; define $f_{\alpha}(d)$ analogously, requiring the angles to be \textit{strictly less than} $\alpha$. Clearly, $f_{\alpha}(d) \leq g_{\alpha}(d)$.

For $\alpha\neq \frac{\pi}{2}$ the problem was considered by Erd\H{o}s and F\"{u}redi \cite{EF}, Li and Otsetarova \cite{LO} and Mironenko \cite{M}. The value $g_{\pi/3}(d) = d+1$ appears in \cite{LO}, \cite{M} and we outline a simplified proof in Section \ref{section:pi/3}. A probabilistic argument for large non-obtuse $\alpha$ and all $d\geq 2$ is examined in \cite{LO}. Explicit approaches for $\alpha = \pi - c$ and $\alpha = \frac{\pi}{3} + c$, for small $c$ and large $d$, can be found in \cite{EF}.

Regarding $\alpha = \frac{\pi}{3} + c$ for sufficiently small $c>0$, the work of Erd\H{o}s and F\"{u}redi \cite{EF} addresses the bounds $(1+c^2)^d \leq f_{\pi/3 + c}(d) \leq (1+4c)^d$. We refine their approach on the lower bound to obtain the asymptotic $f_{\pi/3 + c}(d) = (1+\Theta(c))^d$. (All arguments apply equally well for $g$.) More precisely:
\begin{theorem}
	\label{theorem:exp-bounds}
	For any $c \in (0,1)$, $\delta>0$ and sufficiently
	large $d$,
	\begin{equation*}
	f_{\pi/3 + c}(d) > \left(1+\left(0.0094-\delta\right)c\right)^d.
	\end{equation*} 
\end{theorem}

\begin{theorem}
\label{theorem:exp-upper-bound}
		For any $c \in (0,0.02)$ and sufficiently large $d$,
		$$ f_{\pi/3 + c}(d) < (1+3.75c)^d. $$
	\end{theorem}

The only substantial difference between our arguments and these of Erd\H{o}s and F\"{u}redi is in the analysis of the hypergraph construction discussed in Section \ref{section:embeddinguniformsetsystem}. All other statements and proofs in the paper are not claimed to be original, but we include them for completeness, often with additional details for the sake of readability.

\section{The case $\alpha = \frac{\pi}{3}$}
\label{section:pi/3}

As the sum of angles in a triangle is $\pi$, we immediately have $f_{\alpha}(d) = 2$ for $\alpha \leq \frac{\pi}{3}$ and $g_{\alpha}(d)=2$ for $\alpha < \frac{\pi}{3}$. The following result concerns the value of $g_{\pi/3}(d)$.

\begin{proposition}
	\label{proposition:g-pi/3}
	For all $d\geq 1$, we have $g_{\pi/3}(d) = d+1$. 
\end{proposition}

\begin{proof}
	We simplify the approach in \cite{M}. Let $S=\{a_0, a_1, \ldots, a_n\} \in \mathbb{R}^d$ be a set with $\angle(a_i,a_j,a_k) \leq \frac{\pi}{3}$ for all $i,j,k$. This holds if and only if any triangle $a_ia_ja_k$ is equilateral which occurs if and only if the distances between points in $S$ are equal. Without loss of generality we assume these distances to be $1$. Note also that $\langle a_i - a_j, a_k - a_j \rangle = \cos \frac{\pi}{3} =  \frac{1}{2}$.
	
	For the lower bound, a working construction is the well known regular $d$-simplex -- one way to explicitly describe it in $\mathbb{R}^{d+1}$ is to take the points $a_i = (0,0,\ldots, 1/\sqrt{2},\ldots, 0)$, $1 \leq i \leq d+1$, which lie on the $d$-dimensional hyperplane $\sum_{i=1}^{d+1} x_i = 1/\sqrt{2}$.
	
	To prove the upper bound for $|S|$, it suffices to show the following weaker statement: if $n\geq d$ and $a_0, a_1, \ldots, a_{n}$ are points in $\mathbb{R}^d$ such that
	\begin{equation*}
	\label{equation:60bound}
	|a_i - a_0| = 1, \ \ \ \ \ \ \langle a_i - a_0, a_j - a_0\rangle = \frac{1}{2}
	\end{equation*}
	for all non-zero $i\neq j$, then the vectors $(a_i - a_0)$, $i=1,\ldots,n$ are linearly independent. Indeed, $\dim \mathbb{R}^d = d$ would then imply $n=d$.
	
	So suppose $a_0, a_1, \ldots, a_n$ are as above and assume, for the sake of contradiction, that $(a_i - a_0)$, $i=1,\ldots,n$ are dependent. Then there is a linear combination of them with at least one non-zero scalar -- thus by relabeling the points and dividing by this scalar, if necessary, we may treat $ a_{n} - a_0 = \sum_{i=1}^{n-1} k_i(a_i - a_0)$ for some $k_i\in \mathbb{R}$.
	
	\noindent Taking the inner product with $(a_j - a_0)$ for $j=1,\ldots, n-1$ gives 
	
	\begin{equation*}
	\frac{1}{2} = k_j + \sum_{\substack{i=1 \\ i\neq j}}^{n-1} \frac{k_i}{2} = \frac{k_j}{2} + \sum_{i=1}^{n-1} \frac{k_i}{2} = \frac{k_j}{2} + M,
	\end{equation*}
	
	\noindent and as $M$ does not depend on $j$, we get $k_j = \frac{1}{n}$ for all $j$. Substituting and using the Triangle inequality gives
	\begin{equation*}
	|a_{n} - a_0| \leq \frac{1}{n} \sum_{i=1}^{n-1} |(a_i - a_0)| = \frac{n-1}{n} < 1,
	\end{equation*}
	which is the required contradiction. \qedhere
\end{proof} 

	\section{The construction -- embedding a uniform hypergraph into the cube}
	\label{section:embeddinguniformsetsystem} 
	
	%The construction in this section gives a bound for $\frac{\pi}{3}+c$, which does better than the probabilistic one for the bigger $\alpha$. The only slight disadvantage is that it is formally for large $d$. 
	
	Recall that a hypergraph is \textit{$k$-uniform} if each edge contains exactly $k$ vertices.
	
	We construct greedily a large uniform hypergraph in which every two edges have a tiny overlap. Then we map each edge to a vertex of the hypercube in a way that this overlap results in a small range of possible distances between image points -- this would force the angles to be close to $\pi/3$.   
    
    The difference between the upcoming proposition with its analogue in \cite{EF} is that the number of ``bad choices'' in the greedy algorithm is shown to be smaller, with a slightly stronger argument. However, as we shall see later on, this improvement comes at the price of a more technical treatment.
	\begin{proposition}
		\label{proposition:uniformhypergraph-construction} Fix $c\in (0,1)$ and positive integers $k$ and $d$. There exists a $k$-uniform hypergraph $G$ on $[d] := \{1,2,\ldots, d\}$ with $|F_i \cap F_j| < ck$ for any $F_i, F_j \in E(G)$, $i\neq j$, and $$|E(G)| \geq  \frac{\binom{d}{k}}{\sum_{j=\lceil ck \rceil}^{k}\binom{k}{j}\binom{d-k}{k-j}}.$$ 
	\end{proposition}
	\begin{proof}
		Choose the subset $F_1$ arbitrarily. Now suppose $F_1, \ldots, F_i$ are already chosen. For $k>0$ to be chosen later, consider the set of ``bad'' subsets
		\begin{equation*}
		B_i = \{F \subseteq [d]: |F| = k, |F\cap F_j| \geq \lceil ck \rceil \mbox{ for at least one } j\in [i]\}.
		\end{equation*} 	
        Any $k$-element bad subset can be formed by choosing exactly $j$ of its $k$ elements to be in one of $F_1,\ldots,F_i$ (here $j$ ranges from $\lceil ck \rceil$ to $k$), say $F_m$, and the other $k-j$ to be in $[d]\setminus \{F_m\}$. Hence $|B_i| \leq i\sum_{j=\lceil ck \rceil}^{k}\binom{k}{j}\binom{d-k}{k-j}$. The number of $k$-element subsets of $[d]$ is $\binom{d}{k}$, hence we can choose $F_{i+1} \not\in B_i$ as long as $|B_i| < \binom{d}{k}$. Thus the process can be performed for at least $\binom{d}{k}/\sum_{j=\lceil ck \rceil}^{k}\binom{k}{j}\binom{d-k}{k-j}$ steps. Each step yields an edge of the hypergraph, so we are done. \qedhere %Due to Vandermonde's identity and $ck$ being non-integer, the latter can also be written as $|B_i| \leq i(\binom{d}{k}-\sum_{j=0}^{\lfloor ck\rfloor}\binom{k}{j}\binom{d-k}{k-j})$

%		As any $k$-element bad subset can be formed by choosing $\lceil ck \rceil$ of its $k$ elements from one of $F_1, F_2, \ldots, F_i$ and picking the others arbitrarily, we have $|B_i| \leq  i\binom{k}{\lceil ck \rceil}\binom{d-\lceil ck \rceil}{k-\lceil ck \rceil}$. The number of $k$-element subsets of $[d]$ is $\binom{d}{k}$. Hence we can choose $F_{i+1} \not\in B_i$ as long as $|B_i| < \binom{d}{k}$ -- thus the process can be performed for at least 
%		$$ |E(G)| = \frac{\binom{d}{k}}{\binom{k}{\lceil ck \rceil}\binom{d-\lceil ck \rceil}{k-\lceil ck \rceil}} = \frac{1}{\binom{k}{\lceil ck \rceil}}\frac{d(d-1)\cdots(d-\lceil ck \rceil+1)}{k(k-1)\cdots(k-\lceil ck \rceil+1)} $$
%		steps. With $\binom{n}{x} < (\frac{en}{x})^x$ (here $n=k, x = \lceil ck \rceil$), $\frac{d-\theta}{k-\theta} > \frac{d}{k}$ for $ \theta \in (0,k)$ and $\lceil ck \rceil \geq ck$:
		
%		$$|E(G)| > \bigg(\frac{\lceil ck \rceil}{ek}\bigg)^{\lceil ck \rceil}\bigg(\frac{d}{k}\bigg)^{\lceil ck \rceil} = \bigg(\frac{d\lceil ck \rceil}{ek^2}\bigg)^{\lceil ck \rceil} \geq \bigg(\frac{dc}{ek} \bigg)^{ck}. $$
		
%		\noindent Now for sufficiently large $d$ the choice $k = \frac{dc}{e^2} + f(d)$, where $|f(d)| \leq 1$, is an integer (and maximizes the latter lower bound). It remains to verify 
%		$$ \exp\bigg(\frac{dc^2}{e^2}\bigg) \geq (1+e^{-2}c^2)^d \Leftrightarrow \exp\bigg(\frac{t}{e^2}\bigg) \geq 1+e^{-2}t $$
%		where $t = c^2$. This is exactly the well known inequality $e^w \geq 1 + w$ (here $w = te^{-2}$).
	\end{proof}
Now we focus on estimating the summation in the conclusion of Proposition \ref{proposition:uniformhypergraph-construction}. For two functions $f,g:\mathbb{N}\to\mathbb{R}$ we write $f \gtrapprox g$ if for all $\delta > 0$ and sufficiently large $d$ we have $f(d) - g(d) > -d\delta$ and also let $f\approx g$ if $f\gtrapprox g$ and $g\gtrapprox f$. Observe that if $f\sim g$, that is, $f(d)/g(d) \to 1$ as $d\to \infty$, and also if $d/f(d) \to 0$, then $d/g(d) \to 0$ and $f \approx g$.

\begin{lemma}
    \label{lemma:bullshit-binomial-approx}
    Let $0 < y < x < 1$ be real numbers and $a, b : \mathbb{N}\to \mathbb{N}$ be functions such that $a(n) \approx xn$ and $b(n) \approx yn$. Then $\log \binom{a(n)}{b(n)} \approx n[x\log x - y\log y - (x-y)\log(x-y)]$.
\end{lemma}

\begin{proof}
    By Stirling's approximation we have, as $k,\ell \to \infty$, that $$\binom{k}{\ell} \sim \sqrt{\frac{k}{2\pi \ell(k-\ell)}} \frac{k^k}{\ell^{\ell}(k-\ell)^{k-\ell}}$$
    and hence
    \begin{align*} \log\binom{a(n)}{b(n)} & \approx \frac{1}{2}(\log a(n) - \log b(n) - \log((a(n)-b(n)) - \log 2\pi) \\ & + a(n)\log a(n) - b(n)\log b(n) - (a(n)-b(n))\log((a(n)-b(n)). 
    \end{align*}
    Since $\log n \approx 0$, the terms on the first line contribute overall with $0$. Substituting $a(n) \approx xn$ and $b(n) \approx yn$ in the second line, as well as cancelling terms, now gives the result.
\end{proof}

\begin{proposition}
    \label{proposition:asymptotic}
    With the notation in Proposition $\ref{proposition:uniformhypergraph-construction}$, for any $\delta >0$ and sufficiently large $d$, and for any $c \in (0,1)$, there exists a positive integer $k$ such that
    $$ A(d,k,c) := \frac{\binom{d}{k}}{\sum_{j=\lceil ck \rceil}^{k}\binom{k}{j}\binom{d-k}{k-j}} \geq \left(1+\left(\log 5 - \frac{8}{5}-\delta\right)c\right)^d. $$
\end{proposition}

%\begin{remark}
%The constant $\log 5 - \frac{8}{5} \approx 0.009438$ is indeed positive.
%\end{remark}

\begin{proof}
    We begin with the crude bound $$\sum_{j=\lceil ck \rceil}^{k}\binom{k}{j}\binom{d-k}{k-j} \leq (k-\lceil ck \rceil)\max_{j=\lceil ck \rceil, \ldots,k}\binom{k}{j}\binom{d-k}{k-j} \leq d\max_{j=\lceil ck \rceil, \ldots,k}\binom{k}{j}\binom{d-k}{k-j}.$$   Since $-\log d \approx 0$, we deduce $\log A(d,k,c) \geq \log \binom{d}{k} - \log d - \log \max_{j=\lceil ck \rceil, \ldots,k}\binom{k}{j}\binom{d-k}{k-j} \gtrapprox \log \binom{d}{k} -  \max_{j=\lceil ck \rceil, \ldots,k}\left(\log \binom{k}{j} + \log \binom{d-k}{k-j} \right)$. Let $k\approx ad$ for some $a\in(0,1)$ to be chosen later and write $j\approx bk \approx abd$ where $b\in [c,1]$. From Lemma \ref{lemma:bullshit-binomial-approx} we obtain $$\log \binom{d}{k} \approx d(-a\log a - (1-a)\log(1-a))$$
    and
    $$ \log \binom{k}{j} + \log \binom{d-k}{k-j} \approx d(a\log a - ab\log(ab) - (a-ab)\log(a-ab) $$ $$+ (1-a)\log(1-a) - (a-ab)\log(a-ab) - (1-2a+ab)\log(1-2a+ab)). $$
    This implies $\log A(d,k,c) \gtrapprox d\inf_{b\in [c, 1]}H(a,b)$ where
    $$H(a,b) = %-2a\log a - 2(1-a)\log(1-a) +ab\log(ab) + 2(a-ab)\log(a-ab) + (1-2a+ab)\log(1-2a+ab)$$
    -ab\log a -2(1-a)\log(1-a) + ab\log b + 2a(1-b)\log(1-b) + (1-2a+ab)\log(1-2a+ab).$$
    So we should look for $a\in(0,1)$ and a constant $\rho$ such that $H(a,b) \geq \log(1+\rho c)$ for all $b\in [c,1]$. Furthermore, let us impose the condition $a\leq c$. We compute $\frac{dH}{db} = -a\log a+ a\log b - 2a\log(1-b) + a \log(1-2a+ab)$ and observe that this derivative increases and is zero when $b=a\leq c$ -- hence for $b\in(c,1]$ it is positive and so it remains to establish $H(a,c) \geq \log(1+\rho c)$ for suitable $a\leq c$ and $\rho$. Pick\footnote{Having in mind that the derivative $\frac{dH}{da}$ vanishes at $a=c$ and at $a\sim \frac{c}{5}$, as well as checking the second derivative for these.} $a=\frac{c}{5}$ and note the inequality $H(\frac{c}{5},c) \geq \log(1+(\log 5 - \frac{8}{5})c)$. (Indeed, the ratio $(e^{H(\frac{c}{5},c)}-1)/c$ is strictly increasing and tends to $\log 5 - \frac{8}{5}$ as $c \to 0$.) Therefore $\log A(d,k,c) \gtrapprox d\log(1+(\log 5 - \frac{8}{5})c)$, so for any $\delta_0$ and large $d$ we have $\log A(d,k,c) - d\log(1+(\log 5 - \frac{8}{5})c) + \delta_0 d \geq 0$. By choosing $\delta_0 = \log (1+(\log 5 - \frac{8}{5})c)/(1+(\log 5 - \frac{8}{5}-\delta)c)$, we conclude that $A(d,k,c) \geq (1+(\log 5 - \frac{8}{5}-\delta)c)^d$ for all sufficiently large $d$, as desired.  %We compute $$\frac{dH}{da} = 2\log(1-a) - c\log(a) + (c-2)\log(1 + a(c-2)) + 2 (1-c)\log(1 - c) + c\log(c)$$
    %and $\frac{d^2H}{da^2} = \frac{a(2-c)-c}{a(1-a)(1-a(2-c))}$. The denominator of $\frac{d^2H}{da^2}$ is positive, hence this second derivative has exactly one root.
\end{proof}
    With the next proposition we illustrate how to obtain the large set of points from the large hypergraph.
    \begin{proposition}
        \label{proposition:from-hypergraph-to-hypercube}
        Let $k$ be a positive integer, let $c \in (0,1)$ and let $G$ be a $k$-uniform hypergraph on $[d]$ such that $|F_i\cap F_j| < c k$ for any two distinct $F_i,F_j\in E(G)$. Then there is a set of points in $\mathbb{R}^d$ of size $|E(G)|$ such that any three points form an angle less than $\frac{\pi}{3} + c$.
    \end{proposition}
    	
	\begin{proof}
		%We use the system in Proposition \ref{proposition:uniformhypergraph-construction} with $c = ce \in (0,1)$ (thus $|E(G)| > (1+c^2)^d$).
		Consider the injective map $\theta$ from $E(G)$ to the vertex set of the $d$-dimensional cube $\{0,1\}^d$, which sends $F_i$ to the vertex $v_i$, where the $m$-th coordinate $(v_i)_m$, $1\leq m \leq d$, of $v_i$ is $1$ if $m \in F_i$ and is $0$ otherwise. (In other words, $v_i$ is the characteristic vector of $F_i$.) Note that, by the construction of $G$, each $v_i$ has exactly $k$ coordinates which are $1$. Hence on one hand
		\begin{equation*}
		|v_i - v_j| = \sqrt{|\{1\leq m \leq d: (v_i)_m \neq (v_j)_m \}|} \leq \sqrt{2k}
		\end{equation*} 
	as the total number of $1$s in $v_i$, $v_j$ is $2k$ and they may have no coordinate which is $1$ for both. On the other hand, for $v_i, v_j \in $ Im $\theta$
		\begin{equation*}
		|v_i - v_j| > \sqrt{2k-2c k}
		\end{equation*} 
	since, by construction of $G$, $v_i$, $v_j$ have less than $c k$ coordinates in common, both of which are $1$s. It remains to show that if $v_i, v_j, v_k \in $ Im $\theta$, then $\angle v_iv_jv_k < \frac{\pi}{3} + c$. By the Law of Cosines
	\begin{equation*}
		\cos \angle v_iv_jv_k = \frac{|v_j-v_i|^2 + |v_j-v_k|^2 - |v_i-v_k|^2}{2|v_j-v_i||v_j-v_k|} > \frac{2(2k-2c k) - 2k}{2 \cdot 2k} = \frac{1}{2} - c. 
	\end{equation*} 	
	The result now follows from $\cos(\frac{\pi}{3}+c) - \frac{1}{2}+ c >0$ for $c > 0$.
	\end{proof}
	
	Combinining Propositions \ref{proposition:uniformhypergraph-construction}, \ref{proposition:asymptotic} and \ref{proposition:from-hypergraph-to-hypercube} proves Theorem \ref{theorem:exp-bounds}.
	
	\section{The upper bound}
	\label{section:upperbound-pi/3+c}
	
	In order to obtain a bound from above, we are going to cover our set by a reasonably small ball, project it onto its surface and then bound the cardinality of the image in terms of the minimal distance between points from it. We denote by $\mathbb{S}^{d-1}$ the unit sphere $x_1^2 + x_2^2 + \cdots + x_d^2 = 1$ with center $O$. The \textit{hyperspherical cap} $C(P,\alpha)$ with centre $P \in \mathbb{S}^{d-1}$ and angle $\alpha \leq \pi$ is the set of points $Q\in \mathbb{S}^{d-1}$ such that $\angle POQ \leq \alpha$.
	
\begin{theorem}
    \label{theorem:Rankin-packing}
    {\normalfont$($Rankin $1955$, \cite[Theorem 2]{Ra}$)$} Fix $\alpha \in (0,\frac{\pi}{4})$ to be independent of $d$. Then the maximum possible number of pairwise disjoint hyperspherical caps in $\mathbb{S}^{d-1}$ does not exceed $f(\alpha)$, where $$f(\alpha) \sim \frac{\sqrt{\frac{1}{2}\pi d^3\cos 2\alpha}}{(\sqrt{2}\sin\alpha)^{d-1}} \mbox{ as } d\to \infty.$$
\end{theorem}
	\begin{corollary}
		\label{corollary:kissing}
		Fix $\alpha \in(0,\frac{\pi}{4})$ and let $d$ be sufficiently large. Let $S$ be a pointset on the sphere $\frac{1}{\sqrt{2}}\mathbb{S}^{d-1}$ of radius $\frac{1}{\sqrt{2}}$, such that for some $y\in (0,1)$ and all $P_i, P_j \in S$ we have $|P_iP_j| > 1-y$. Then $|S| \leq d^2(1-y)^{-d}$.	
	\end{corollary}
	
\begin{proof} 
    Begin by projecting $S$, through the origin, onto the unit sphere $\mathbb{S}^{d-1}$ -- let $Q_i$ be the image of $P_i$. Then $|Q_iQ_j| > (1-y)\sqrt{2}$ and so $\sin \frac{\angle Q_iOQ_j}{2} > \frac{1-y}{\sqrt{2}}$. In particular, with $\alpha = \arcsin(\frac{1-y}{\sqrt{2}}) \in (0, \frac{\pi}{4})$, the hyperspherical caps $C(Q_i,\alpha)$ are pairwise disjoint and so Theorem \ref{theorem:Rankin-packing} gives, for sufficiently large $d$, that $|S| \leq 2f(\alpha)$ and hence that $|S| \leq d^2(1-y)^{1-d}$.
\end{proof} 
    
    Our tool of covering the pointset by a small enough ball is as follows.    
    \begin{theorem}
        \label{theorem:Jung}
		{\normalfont (Jung 1901, \cite{Jung})} Any set $A \subseteq \mathbb{R}^d$ of diameter $1$ is contained in a closed ball of radius $\sqrt{\frac{d}{2(d+1)}}$.
    \end{theorem}
    
    Before moving on to the main proof, we include separately a small triangle lemma.
    \begin{lemma}
        \label{lemma:triangle-sozopol} Let $ABC$ be an isosceles triangle $(|AC|=|BC|)$ and assume there exist points $M \in AC$ and $N \in BC$ such that $|MN|>|AC|$. Then $|AB| > |MN|$. 
    \end{lemma}
    \begin{proof}
        Without loss of generality treat $|CM| \geq |CN|$. Let the line through $A$, parallel to $MN$, intersect the segment $BN$ at the point $K$ (which may coincide with $B$) and the line through $M$, parallel to $BC$, intersect $AK$ at the point $P$. Then $|AK| > |PK| = |MN|$ (since $PKNM$ is a parallelogram) and now the condition of the lemma gives $|AK| > |AC| = |BC| > |CK|$. In particular $\angle ACK > 60^{\circ}$ as it is the largest angle in triangle $ACK$, hence $\angle AKB$ is the largest angle in triangle $AKB$. Therefore $|AB| > |AK| > |MN|$.
    \end{proof}
    
    We are now ready to prove the upper bound.	
	
	\begin{proof}
	    [Proof of Theorem $\ref{theorem:exp-upper-bound}$]
		%We fill in the details in \cite[p.280]{EF}. 
		Consider a pointset $S = \{P_1, P_2, \ldots,P_{|S|}\}$ with angles less than $\frac{\pi}{3} + c$. In any triangle $P_iP_jP_k$ the largest angle is less than $\frac{\pi}{3} + c$ and the smallest angle exceeds $\frac{\pi}{3} - 2c$, hence the ratio of the smallest and largest sides, by the Sine Law, exceeds
		$$ \frac{\sin(\frac{\pi}{3} - 2c)}{\sin(\frac{\pi}{3} + c)} \geq 1 - 1.744c $$
		for $c\in (0,0.024)$. 
		
		Next, we may assume that $|P_iP_j| = 1$ is the largest distance in $S$. We now bound the smallest distance in $S$, say $|P_kP_l|$. If without loss of generality $P_kP_j$ is the largest side in $\triangle P_jP_kP_l$, then $|P_kP_l| \geq (1-1.75c)|P_kP_j|$ and in $\triangle P_iP_kP_j$ we have $|P_kP_j| \geq (1-1.744c)|P_iP_j|$. Hence
		$$ |P_kP_l| \geq (1-1.744c)^2|P_iP_j| > 1-3.488c > \frac{1}{\sqrt{2}}. $$ 
		
		On the other hand, by Theorem \ref{theorem:Jung} there is a closed ball containing $S$ and of radius $\sqrt{\frac{d}{2(d+1)}} < \frac{1}{\sqrt{2}}$ -- so there is such a ball $B$ (centered at $O$, say) with radius $\frac{1}{\sqrt{2}}$. For $i=1,\ldots,m$ denote by $Q_i$ the intersection of the line connecting $O$ and $P_i$ with the sphere $\partial B$ and note that $P_i$ is between $O$ and $Q_i$. 
		As $|P_mP_n| > \frac{1}{\sqrt{2}} = |OQ_i|$, by Lemma \ref{lemma:triangle-sozopol} it follows that $|Q_mQ_n| > |P_mP_n|$ for all $m$, $n$ and hence $|Q_mQ_n| > 1-3.5c$. Finally, by Corollary \ref{corollary:kissing} and $c < 0.02$ (so that $(1+3.75c)(1-3.488c) > 1$), we conclude 
		$$ |S| < d^2(1-3.488c)^{-d} < (1+3.75c)^d  $$ 
		for sufficiently large $d$.
	\end{proof}
    
\textbf{Acknowledgments.} This research was supported by an undergraduate course for projects at the University of Oxford. The author would like to thank the supervisor David Conlon for introduction to the problem, as well as for helpful discussions. Thanks also to team Germany South from the 2020 edition of the International Tournament of Young Mathematicians, for discussions on the technical Proposition \ref{proposition:asymptotic}, and to Lyuben Lichev, for remarks on the details in the paper. 

%\section{Appendix}
%Here is the code that we used to compute the power indices. 
%\lstinputlisting[language=Python]{code.py}

\end{document}